\documentclass[11pt]{article}
\usepackage{amsmath}
\usepackage{amsfonts}
\usepackage{amssymb}
\usepackage{amsthm}
\usepackage{graphicx}
\usepackage{mathrsfs,eucal,xspace}
\usepackage[all]{xy}
\usepackage{pstricks,pst-node}
\usepackage{fullpage}

\newtheorem{theorem}{Theorem}

\newtheorem{definition}[theorem]{Definition}
\newtheorem{lemma}[theorem]{Lemma}

\theoremstyle{remark}
\newtheorem*{example}{Example}

\newcommand{\spam}{\mathop{\mathrm{span}}}

\newcommand{\supp}[1]{{\text{supp}}({#1})}
\newcommand{\ints}{\mathbb{Z}}
\newcommand{\reals}{\mathbb{R}}

\newcommand{\dif}{\mathrm{d}}

\newcommand{\N}{\mathcal{N}}

\title{{On Local RBF Approximation}
}
\author{T.Hangelbroek\thanks{Texas A\&M
    University College Station, TX 77843, USA. Research supported
    by NSF Postdoctoral Research Fellowship.} }
\begin{document}
\maketitle
\begin{abstract}
The purpose of this paper is to investigate RBF approximation with highly nonuniform
centers. Recently, DeVore and Ron have developed a notion of the local density of a 
set of centers -- a notion that permits precise pointwise error estimates 
for surface spline approximation.  We give an equivalent, alternative characterization
of local density, one that allows effective placement of centers at different resolutions. 
We compare, also, the pointwise results of DeVore--Ron to previously works of
 Wu and Schaback  and of  Duchon.
\end{abstract}
\section{Introduction}
This brief article concerns local approximation results for radial basis function (RBF) approximation with the goal of effectively placing centers at varying resolutions. 
We consider RBF approximants of the form
$x\mapsto \sum_{\xi\in\Xi} \phi(x-\xi)$, 
where the arrangement of centers, $\Xi$, may be highly nonuniform.
A motivation for this set up is that centers may be placed strategically
to treat defects in the target function. 
This becomes very important in high dimensions, where conventional quasi-uniform placement of centers is extremely costly; 
error estimates assuming a (small) fill distance $h$ require a placement of $\mathcal{O}(h^{-d})$ centers; obtaining a comparable error with fewer centers
is clearly desirable.
To this end, we seek a method by which $\Xi$ can be chosen 
to achieve a pointwise error that reflects the local arrangement of $\Xi$.
In turn, this requires finding a useful measure of 
the {\em local  density} of $\Xi$. 

In \cite{DeRo}, DeVore and Ron 
establish powerful local error estimates for kernel based approximation;
along the way they develop a satisfactory notion of local density -- the {\em majorant}
-- expressed as a function over a domain containing the centers 
(see (\ref{majorant}) below). 
This function gives, roughly, the distance to the nearest unisolvent subset of $\Xi$.
However, it also satisfies an extra condition of global compatibility: 
it may not grow or shrink too rapidly. This condition is not stated explicitly, rather 
it is contained in the definition, but it plays an essential role in their local error estimates.
In this article we investigate this compatibility condition and  give an equivalent condition
that lends itself to effective placement of highly nonuniform centers. 

Error estimates in \cite{DeRo} show that kernel approximation bears a strong similarity to univariate spline approximation, which exhibits local convergence in the sense that error decays rapidly over intervals where breakpoints are tightly spaced. Indeed, 
such results have long been known for spline quasi-interpolation, at least since \cite{BF}. 
If $Q_T$ is the quasi-interpolation operator of 
order $r$ associated with knots $T= (t_j)$, then \cite[(4.18)]{DeLo} in conjunction
with a theorem of Whitney \cite[Theorem 4.2]{DeLo} tells us that for
 $x\in [t_j,t_{j+1}]$,
$$ \left|f(x) - Q_{T}(f) (x)\right|\le C_r 
\omega_{r+1}(f,t_{j+r}-t_{j-r+1}).$$
We note that the error from spline approximation at $x$ depends only on the distribution of knots near $x$ (the $r^{\mathrm{th}}$ {\em nearest neighbors}) and the smoothness of the target function in a neighborhood of $x$. This result can be attributed to the local nature of the  B-splines basis (and its associated dual functionals).

Because kernels are often globally supported, 
and because a truly local basis similar to the B-splines seems to be out of the question, the distance to the nearest neighbors may be unsuitable for measuring local density.
Kernel approximants exhibit {\em far field} effects, meaning that each kernel has a global influence, and the majorant of DeVore \!--\! Ron, via the global compatibility condition, penalizes remote, sparse density in an effort to mitigate such far field effects.

Remarkably, it is sometimes the case that RBF interpolation is local in a stronger sense than considered by DeVore \!--\! Ron. Specifically, rates of convergence for interpolation of certain target functions (those coming from the {\em native space}) with certain RBFs can be measured in terms of an expression that takes into account only the local distribution of centers.
However, the drawback is that such results are presently only known to hold for 
target functions from a specific class.

In the following section, we present an alternative, equivalent characterization of the
majorant of Devore and Ron, one that involves a global compatibility condition, and we give a self-contained development of their pointwise error estimates that explicitly uses this compatibility condition. In Section 3, we demonstrate how the compatibility
condition may be used to place centers nonuniformly. Section 4 is a discussion of (previously known) local error estimates in the native space context.

\section{Local Estimates}
In \cite{DeRo}, DeVore and Ron construct a local approximation scheme using a simple 
measure of the local density of centers. 
This initial notion of density (given in Definition \ref{LD}) is not suitable to capture far field effects, and  the analysis of the scheme's convergence eventually relies on a more refined notion of local density: the initial density's {\em majorant}.  
In this section, we recast the DeVore \!--\! Ron result with a different but equivalent local density parameter, one that lends itself to efficient placement of centers.
As in \cite{DeRo}, 
we begin by giving the initial definition of local density,
a function on $\reals^d$ which at each $\alpha\in\reals^d$ 
indicates a radius sufficient to capture a $K$-stable local polynomial 
reproduction of order $\ell$: 
\begin{definition}[Local Density] \label{LD}
Given a set of centers $\Xi \subset \reals^d$, a \emph{local density (LD)}  $\rho:\reals^d \to \reals_+$ is a function with an associated {\it local polynomial reproduction} of precision $\ell$. That is to say, there is a  kernel
$a:\Xi\times \reals^d \to \reals:(\xi,\alpha)\mapsto a(\xi,\alpha)$
so that the following hold:
\begin{description}
\item[\bf  (Support)]  For $|\xi - \alpha|>\rho(\alpha)$, $a(\xi,\alpha) = 0$.
\item[\bf (Precision)] For all $p\in \Pi_{\ell}$ we have $\sum_{\xi\in\Xi} a(\xi,\alpha)p(\xi) = p(\alpha)$.
\item[\bf (Stability)] There is $K>0$ such that $\sum_{\xi\in\Xi} |a(\xi,\alpha)|<K$ for all $\alpha$.
\end{description}
\end{definition}
We note that this definition is given, in more or less the same form, by Wu and Schaback in \cite[Lemma 2]{WuSch} (we discuss their local results in Section 4.2).

This construction 
allows the surface spline, $\phi$
\begin{equation}\label{surface}
\phi(x) :=\phi_k(x):= C_{k,d}
\begin{cases}
 |x|^{2k-d}\log |x| &\quad \text{for even}\ d\\
 |x|^{2k-d} &\quad \text{for odd}\ d
\end{cases}
\end{equation}
 (also known as polyharmonic splines 
because they are the fundamental solution of the $k$-fold Laplacian $\Delta^k$), 
to be approximated by a linear combination of nearby shifts. 
This is accomplished with a small, local error:
$$|\phi(x- \alpha) - \sum a(\xi,\alpha) \phi(x-\xi)| \le C\, \rho(\alpha)^{2k -d} \left(1+\frac{|x-\alpha|}{\rho(\alpha)}\right)^{-\nu},\quad \nu := \ell+d - 2k,$$
and one can generate the approximant:
$$T_{\Xi} f(x) := \int_{\reals^d} \Delta^{k}f(\alpha) \sum_{\xi\in \Xi} a(\xi,\alpha) \phi(x-\xi) \, \dif \alpha.$$
This leads to convenient pointwise error estimates:
\begin{equation}\label{firsterror}
|f(x) - T_{\Xi} f(x)| \le C \int_{\reals^d} |\Delta^{k}f(\alpha)| \, \rho(\alpha)^{2k -d} \left(1+\frac{|x-\alpha|}{\rho(\alpha)}\right)^{-\nu} \, \dif \alpha.
\end{equation}

Observe that the construction of the approximation operator $T_{\Xi}$ depends only on a local polynomial reproduction
$(\xi,\alpha)\mapsto a(\xi,\alpha)$. The estimate (\ref{firsterror}) holds for any LD with local polynomial reproduction $a$. In particular, it holds for any function $\rho'$ with $\rho'\ge\rho$.  With this in
mind, it may seem tempting to use an optimally small $\rho$: the LD so that $\rho(\alpha)$ is the minimal radius around $\alpha$ that captures a $K$-stable polynomial reproduction of order $\ell$.
Sadly, because of the global nature of the kernels, this is unsuitable for producing precise estimates.

In  \cite{DeRo} 
the LD is used to create a ``majorant'':
\begin{equation}
\label{majorant}
H(x) = \sup_{y \in \reals^d} \rho(y)\left(1+\frac{|x-y|}{\rho(y)}\right)^{-r}. 
\end{equation}

In short, the coefficient kernel associated to the LD is used to create the approximant, 
but, in order to attack the estimate (\ref{firsterror}), 
 an expression involving the majorant controls the error. 
The initial LD only plays an ancillary role: to construct the approximant via the coefficient kernel and to generate the majorant. 

To be sure, the entire exercise could be repeated using only $H$ and entirely without $\rho$.
It is a simple task to show that $H$ is itself an LD (since $H(x)\ge \rho(x)$ and the local polynomial reproduction $(\xi,x)\mapsto a(\xi,x)$ is a local polynomial reproduction for $H$ as well). Moreover, the majorant of $H$ is simply a constant multiple of $H$ (a constant depending only on $r$).
It follows that by replacing the initial LD $\rho$ with its majorant $H$,
one would obtain the same results. 

Alternately, from the beginning one may impose the assumption that the LD is equivalent to
its majorant (i.e., it is {\em self-majorizing}).
This is a change of perspective: 
from the point of view that centers have been given outside of our control 
(and with the goal of remaining faithful to the local distribution of data by having error estimates reflecting the local density)  
to the setting where the spacing of centers is chosen to reflect regions of interest or to attack defects in the target function.
We will choose centers having an LD that is correct from the start, having an extra condition designed to handle far field effects.
The condition on the LD will be different from that of \cite{DeRo}; 
it is a slightly more easily verified property: {\em slow growth}.
To proceed, we formalize both extra assumptions (slow growth and self-majorization) 
under the heading of global compatibility and take a moment to discuss their equivalence.
\begin{definition}[Global Compatibility]\label{SG}
If there is a constant $C_{sg}>0$ so that for every $x$ and $\alpha$, we have
\begin{equation}\label{sg}
\rho(\alpha) \le C_{sg}\rho(x)\left(1+\frac{|x-\alpha|}{\rho(x)}\right)^{1-\epsilon},
\end{equation}
we say $\rho$ exhibits $1-\epsilon$ {\em slow growth}.

The function $\rho:\Omega\to \reals_{+}$ exhibits {\em self-majorization} of order $r$ if there is a constant $C_{sm}>0$ so that for every $x$ and $y$, we have
\begin{eqnarray}
\label{sm}
\rho(y) \ge C_{sm} \rho(x)\left(1+\frac{|x-y|}{\rho(x)}\right)^{-r}
\end{eqnarray}
\end{definition}
The equivalence of these two assumptions can be expressed formally:
%
%
\begin{lemma}
If $\rho$ satisfies the property of self-majorization (\ref{sm}) 
then it satisfies the property of slow growth (\ref{sg}) with $\epsilon = \frac{1}{r+1}$ and constant $C_{sg}$ depending only on $r$ and $C_{sm}$. 
Likewise, if $\rho$ satisfies (\ref{sg}) then it satisfies (\ref{sm}) with 
$r = \frac{1-\epsilon}{\epsilon}$ 
and constant $C_{sm}$ depending only on $\epsilon$ and $C_{sg}$.
\end{lemma}
\begin{proof}
When $|x-\alpha|\ge \rho(x)$, (\ref{sg}) implies that 
$\rho(\alpha) \le 2^{1-\epsilon}C_{sg} \rho(x)^{\epsilon}|x-\alpha|^{1-\epsilon}$,
so 
$$\frac{2^{\epsilon-1}}{C_{sg}}\rho(\alpha)^\epsilon\left(\frac{|x-\alpha|}{\rho(\alpha)}\right)^{\epsilon-1}\le  \rho(x)^{\epsilon}.$$
On the other hand, when  
$|x-\alpha|< \rho(x)$, $\rho(\alpha) \le 2^{1-\epsilon}C_{sg} \rho(x)$, so
$\rho$ satisfies self-majorization with $C_{sm} = \min\left(\frac{2^{\epsilon-1}}{C_{sg}},\left(\frac{2^{\epsilon-1}}{C_{sg}}\right)^{1/\epsilon}\right)$.

When $|x-\alpha|\ge \rho(x)$, (\ref{sm}) implies that 
$\rho(\alpha) \ge 2^{-r}C_{sm} \rho(x)^{1+r}|x-\alpha|^{-r}$,
so 
$$\frac{2^{r}}{C_{sm}}\rho(\alpha)^{1+r}\left(\frac{|x-\alpha|}{\rho(\alpha)}\right)^{r}\ge  \rho(x)^{1+r}.$$
On the other hand, when  
$|x-\alpha|< \rho(x)$, 
$\rho(\alpha) \le 2^{-r}C_{sm} \rho(x)$, so
$\rho$ satisfies slow growth with
$C_{sg} = \max\left(\frac{2^{r}}{C_{sm}},\left(\frac{2^{r}}{C_{sm}}\right)^{1/(1+r)}\right)$.
\end{proof}

Either of these extra assumptions on $\rho$ are sufficient to obtain the error estimate
in \cite{DeRo}:
\begin{theorem}[DeVore Ron I] \label{DeRoI}
Let $\ell>2k-d+1$. Suppose that $\rho$ satisfies Slow Growth with $\epsilon> \frac{2k}{\ell}$. There is a constant $C$ so that for $f\in C^{2k}(\reals^d)$ having
compact support,
$$|f(x) - T_{\Xi} f(x)|\le C \rho(x)^{2k} \|\Delta^k f\|_{\infty}$$
\end{theorem}
\begin{proof}
This follows by applying the growth assumption to (\ref{firsterror}) and writing $\gamma = 1-\epsilon$ to obtain:
\begin{eqnarray*}
|f(x) - T_{\Xi} f(x)| &\le& C \int_{\reals^d} |\Delta^{k}f(\alpha)| \rho(x)^{2k -d} \left(1+\frac{|x-\alpha|}{\rho(x)}\right)^{\gamma(2k-d)} \left(1+\frac{\frac{|x-\alpha|}{\rho(x)}}{\left(1+\frac{|x-\alpha|}{\rho(x)}\right)^{\gamma}}\right)^{-\nu} \, \dif \alpha\\
&\le& C \rho(x)^{2k -d}  \|\Delta^{k}f(\alpha)\|_{\infty} \int_{\reals^d}  \left(1+\frac{|x-\alpha|}{\rho(x)}\right)^{2k-d-\ell+\ell\gamma} \, \dif \alpha\\
&\le& C \rho(x)^{2k} \|\Delta^{k}f(\alpha)\|_{\infty} \int_0^{\infty} (1+R)^{ 2k-d-\ell+\ell \gamma}R^{d-1} \dif R.
\end{eqnarray*}
The second inequality follows by writing $(2k-d)\gamma - \nu(1-\gamma) =(2k-d)\gamma- (\ell- 2k +d)(1-\gamma)$.
The convergence of the last integral is a consequence of the assumption  $\gamma<1 - \frac{2k}{\ell}$.
\end{proof}
A further result from \cite{DeRo}, is that functions of lower smoothness can also be treated with local error estimates. The operator $T_{\Xi}$ is instrumental in obtaining low smoothness results, albeit indirectly. 
This is the point of  \cite[Theorem 5.3]{DeRo}, which, for completeness, we rephrase in a simplified form as Theorem \ref{lower} in terms of the slow growth assumption.
The lower order result is technically more complicated than that of Theorem  \ref{DeRoI}. 
It is a common technique to use interpolation theory to obtain direct approximation results 
for functions of lower smoothness.  
DeVore and Ron use an argument of this type, that splits $f$ into 
a rough but benign part, $b$, and a smooth part, $g$. 
This is done in a way that is not entirely straightforward, by controlling the size of 
$b$ (and the smoothness of $g$)  in a precise way to match the local density. 

To discuss lower smoothness estimates, we first introduce fractional smoothness spaces. These can be expressed in numerous different ways: as Besov spaces, Triebel-Lizorkin spaces or (more familiarly) H{\" o}lder-Zygmund spaces. In the setting we consider, these are the same spaces. That is, we consider $F_{\infty,\infty}^{\sigma}(\reals^d) = B_{\infty,\infty}^{\sigma}(\reals^d)=C^{({\sigma})}(\reals^d)$. However, the exact smoothness norm we employ is the $B_{\infty,\infty}^{\sigma}$ norm, defined in terms of wavelet coefficients.  Smooth functions can be expanded as $f= \sum_{j=0}^{\infty}c_j \psi_j$
and the smoothness seminorm is expressed in terms of coefficients $c_j$.

A totally standard construction, used also in \cite{DeRo}, indexes wavelets by gendered, dyadic cubes: $\nu \in \mathcal{D}$, where each $\nu = (e_{\nu},I_{\nu})$ is a pair comprising: 
\begin{itemize}
\item a {\em gender} $e= e_{\nu} \in \{0,1\}^{d}\setminus \{\mathbf{0}\}$
\item and a {\em dyadic cube} $I :=I_{\nu} = 2^{-j}(k+[0,1]^d)$.
\end{itemize}
For a general dyadic cube of this form we denote
 the corner by $c(I):=2^{-j}k$ 
 and the side-length by $\ell(I):=2^{-j}$. 
 These definitions extend for gendered cubes: 
 $c(\nu):=c(I_{\nu})$ and $\ell(\nu) := \ell(I_{\nu})$.
 Under this indexing, each gendered cube $\nu$ has exactly one parent $\nu'$, 
where $I_{\nu} \subset I_{\nu'}$, $e_{\nu}=e_{\nu'}$ and $\ell (\nu') = 2 \ell(\nu)$.

The wavelet system we employ is a family of $C^{r}$, compactly supported functions, with
$r>2k$.
Each wavelet is related to one of $2^d -1$ prototypes by affine changes of variable: 
$\psi_{\nu}(x) = \Psi_{e_{\nu}}\left(\frac{x-c(\nu)}{\ell(\nu)}\right)$. 
In other words, each wavelet is a translated, rescaled copy of one of the $2^d-1$ functions $\Psi_{e}\in C^{r}(\reals^d)$. 
Consequently the supports of wavelets are obtained by affine transformations, and each is contained in a ball with radius proportional to the side-length and  centered at the corner of the cube $I(\nu)$. I.e., there is $\Gamma>0$ so that for all $\nu\in \mathcal{D}$
$$\overline{I_{\nu}}:= \supp{\psi_{\nu}} = c(\nu)+\ell(\nu)\times \supp{\Psi_{e_{\nu}}}\subset
B\bigl(c(\nu),\Gamma \ell(\nu)\bigr).$$
For orthogonal wavelet systems, compactly supported continuous functions have the unique expansion $f = \sum_{\nu\in \mathcal{D}} f_{\nu} \psi_{\nu}$. 
The smoothness seminorm of $f$ is
 $$|f|_{B_{\infty,\infty}^{\sigma}}:= 
\sup_{\nu\in\mathcal{D}} \left(\ell(\nu)^{-{\sigma}}\, |f_{\nu}|\right).$$
%
%
\begin{theorem}[DeVore Ron II]\label{lower} Let $\ell>2k-d+1$. Suppose that $\rho$ satisfies Slow Growth with $\epsilon > \frac{2k}{\ell}$. There is $C>0$ so that for all 
compactly supported $f\in B_{\infty,\infty}^{\sigma}$, $\sigma<2k$, 
there is $s_{f,\Xi} \in \spam(\phi,\Xi)$ so that
$$|f(x) - s_{f,\Xi}(x)| \le C \rho(x)^{\sigma} \| f\|_{B_{\infty,\infty}^{\sigma}}.$$
\end{theorem}
\begin{proof}
We split $f = g+b$ where $|b(x)|\lesssim \rho(x)^{\sigma} |f|_{B_{\infty,\infty}^{\sigma}}$
and $|\Delta^{k}g(x)| \lesssim \rho(x)^{\sigma - 2k} |f|_{B_{\infty,\infty}^{\sigma}}.$  A consequence of this and Theorem \ref{DeRoI} is that 
$|f(x) - T_{\Xi}g(x)| \lesssim \rho(x)^{\sigma} \|f\|_{B_{\infty,\infty}^{\sigma}}$ 
and the theorem follows with $s_{f,\Xi} = T_{\Xi}g$.

To obtain the split, partition $\mathcal{D} = \mathcal{D}_g\cup \mathcal{D}_b$ 
by selecting cubes $\nu$ according to the density over $\overline{I}_{\nu}$:
$$\nu \in \mathcal{D}_g \quad \text{iff} \quad \ell(\nu)\ge \rho(\nu):=
\max_{y\in \overline{I}_{\nu}} \rho(y).$$
%
Define 
$g:= \sum_{\nu \in \mathcal{D}_g}f_{\nu}\psi_{\nu} =  \sum_{\ell(\nu)\ge\rho(\nu)}f_{\nu}\psi_{\nu}.$ Estimating the iterated Laplacian of a wavelet is  straightforward: for $x\in\overline{I}_{\nu}$, 
$|\Delta^k \psi_{\nu}(x)|
\le 
C
\ell(\nu)^{-2k}.$ 
Consequently, 
$$|\Delta^k g(x)| 
\le 
 C\sum_{\substack{\nu \in \mathcal{D}_g\\ x\in \overline{I_{\nu}}}} 
  f_{\nu} \ell(\nu)^{-2k}
\le 
C \, |f|_{B_{\infty,\infty}^{\sigma}}\, 
\sum_{\substack{\nu \in \mathcal{D}_g\\ x\in \overline{I_{\nu}}}}
   \ell(\nu)^{\sigma-2k} 
   \le C'  |f|_{B_{\infty,\infty}^{\sigma}}(\rho(x))^{\sigma-2k}.$$
The final estimate deserves some explanation.  Note that $\nu \in \mathcal{D}_g$ and $x\in \overline{I_{\nu}}$ imply that $\ell(\nu)\ge \rho(x)$. 
Finding $j\in \ints$ so that $2^j\ge\rho(x)>2^{j-1}$,  the number of
wavelets with $\ell(\nu)=2^j$ that have $x$ in their support is bounded, 
$\#\{\nu:\ell(\nu)=2^j, x\in\overline{I}_{\nu}\} \le N$, with a constant independent of $x$ and $j$. 
We may rewrite the sum in the next to last expression in the chain of inequalities as
$$
\sum_{\substack{\nu \in \mathcal{D}_g\\ x\in \overline{I_{\nu}}}} 
   \ell(\nu)^{\sigma-2k}  
\le
\sum_{\substack{\ell(\nu)=2^j\\ x\in\overline{I}_{\nu}}} 
\sum_{i=0}^{\infty}
  \left(2^{j+i}\right)^{\sigma-2k}
\le N 2^{j(\sigma-2k)} \sum_{i=0}^{\infty}\left(2^{i}\right)^{\sigma-2k}
\le C \bigl(\rho(x)\bigr)^{\sigma-2k}.
  $$

Estimating the size of 
$b(x):= \sum_{\nu \in \mathcal{D}_b}f_{\nu}\psi_{\nu} $,
we write 
$|b(x)|\le \sum_{\nu \in \mathcal{D}_b, x\in \overline{I}_{\nu}}|f_{\nu}|$, 
which is bounded by 
$|f|_{B_{\infty,\infty}^{\sigma}} \sum_{\nu \in \mathcal{D}_b,x\in \overline{I}_{\nu}} \bigl(\ell(\nu)\bigr)^{\sigma}$.
If $x$ and $y$ are in $\overline{I}_{\nu}$ and if $\rho(y) > \ell(\nu)$ then  
$$\rho(x)
\ge 
C_{sm}\rho(y) \left(1+\frac{|x-y|}{\rho(y)}\right)^{-r}
\ge 
\ell(\nu)C_{sm}(1+2\Gamma)^{-r}\quad
\Rightarrow\quad
\ell(\nu) \le C \rho(x)
$$ 
As in the case of $g$, it follows that the series
$|b(x)|\le 
|f|_{B_{\infty,\infty}^{\sigma}} 
\sum_{\ell(\nu) \le C\rho(x), x\in\overline{I}_{\nu}}
 \bigl(\ell(\nu)\bigr)^{\sigma}$
can be rewritten as a sum of geometric series, to obtain
$|b(x)|\le C\, |f|_{B_{\infty,\infty}^{\sigma}} \, \bigl(\rho(x)\bigr)^{\sigma}$.
\end{proof}

\section{Placing Centers}
We now turn to a discussion of how condition (\ref{sg}) may be directly implemented to produce effective global approximants with local error estimates.
In this section we present an algorithm for  generating a set of centers with a fixed spacing,
having a more refined spacing on a particular subset. 
This can be done in such a way that the pointwise error from surface spline approximation using $\phi_k$ (the surface spline of order $k>d/2$), reflects the local arrangement of centers. 

We begin with a compact set $\Omega$ (e.g., a finite set
or some lower dimensional manifold) in which we wish to 
place centers with increased density (say a  spacing of $h^s$, $s>1$). 
Furthermore, we wish to have an ``ambient'' density of $h$ outside of $\Omega$. 
Without loss, we assume $h= 2^{-j}$ and
$h^s = 2^{-js}$.

Invoking a $1-\epsilon$ slow growth condition from Definition \ref{SG} (a range of
valid $\epsilon$'s will be determined momentarily), 
with $\rho(\alpha) = 2^{-sj}$ and $\rho(x) = 2^{-j}$, 
we see that extra centers must be placed in a region $\tilde{\Omega}$ with
$\Omega\subset \tilde{\Omega} \subset \{x\mid \mathrm{dist}(x,\Omega) \le 2^{-j(\frac{1-s\epsilon}{1-\epsilon})}\}.$ This imposes a certain condition on the relationship between $s$, the slow growth parameter $\epsilon$, and ultimately the polynomial precision $\ell$ (by the conditions of Theorem \ref{DeRoI}).

To force the set $\tilde{\Omega}\setminus \Omega$ to shrink with $h$, 
$\epsilon\times s$ should be less than $1$; the smaller this product is, the smaller the region of extra centers will be. In turn, Theorem \ref{DeRoI} forces $\ell>2k/\epsilon$.

We place extra, gridded centers with dyadic spacing in annular regions around $\Omega$.
That is, we identify  a sequence of annular regions $\Omega_0, \dots, \Omega_{j_0}$ (with $j_0 = sj-j-1$) around $\Omega$. In each region $\Omega_J$ we place centers $\Xi_J$ having constant spacing, and this spacing diminishes the farther $\Omega_J$ is from 
$\Omega$. Specifically, the $J$th region, $\Omega_J$, has centers with  spacing $2^{-js+J}$. 
The union of these sets $\bigcup_{J=0}^{sj-j-1} \Omega_J$ is $\widetilde{\Omega}$.


{\bf Initial step:} The first such set, $\Omega_0$, contains $\Omega$ and has centers with spacing 
$2^{-sj}$. We make it slightly larger, in order to ensure that sufficiently many centers are present to satisfy Definition \ref{LD}.
Thus we set 
$$\Omega_0 = \{x \mid \mathrm{dist}(x,\Omega)\le \ell\}\quad \text{and} \quad \Xi_0 = 2^{-sj} \ints^d\cap \Omega_0 $$
Inside $\Omega_0$ we have placed a set of  gridded centers
with spacing  
$2^{-sj}$. 

From this, we have
$\rho(\alpha)  = \ell \,2^{-sj}$ for $\alpha \in \Omega$. Indeed, for $\alpha\in\Omega$, there is
a simplex $S_{\alpha}$ containing $\alpha$ and  contained in $\Omega_0$, with side-length 
$\ell 2^{-sj}$ and corners in $\Xi_0$. 
The points $S_{\alpha} \cap \Xi_0$ are in general position for interpolation by $\Pi_{\ell}$, and the associated Lagrange functions
for polynomial interpolation $l_{\xi}(\alpha)$ give the required local polynomial reproduction $a(\xi,\alpha) = l_{\xi}(\alpha)$. The stability constant $K$ is none other than the Lebesgue constant
for this interpolation problem, which is bounded by ${2\ell-1 \choose \ell}$, as demonstrated in \cite[Theorem 2.2]{Bos}. 

{\bf General step:}
As we did before, the width of $\Omega_J$ can be determined from the slow growth condition.
Setting
$$\Omega_J := \left\{x\mid \mathrm{dist}(x,\Omega) \le \ell\, 2^{\left(\frac{J}{1-\epsilon} -js\right)}\right\}\setminus \bigcup_{\nu=0}^{J- 1}\Omega_{\nu}\quad \text {and} \quad
\Xi_J := \Omega_J \cap 2^{J-sj} \ints^d$$
guarantees that
$\rho(\alpha)  = \ell \,2^{J-sj}$ in $\Omega_{J-1},$ since $\alpha \in\Omega_{J-1}$ is at the
center of a ball of radius $\ell 2^{J-sj}$ contained in $\bigcup_{\nu=0}^{J}\Omega_{\nu}$. As before,
there is a simplex $S_\alpha$ containing $\alpha$ contained in this ball, and the corresponding
Lagrange functions give the required coefficient kernel.

{\bf Verifying the slow growth condition:}
It follows that for $x \in \Omega$ and $\alpha \in \Omega_J$ the distance
$\mathrm{dist}(x,\alpha)\ge \ell 2^{\frac{J-1}{1-\epsilon}-js}$, and
$$2^{J-sj} =  2^{-sj+1} \left(\frac{2^{\frac{J-1}{1-\epsilon}-sj}}
{2^{-sj}}\right)^{1-\epsilon}
\qquad
\Longrightarrow
\qquad
\rho(\alpha) \le 2 \rho(x) \left(1+ \frac{\mathrm{dist}(x,\alpha)}{\rho(x)}\right)^{1-\epsilon}
$$
Likewise, for $\alpha \in \Omega_{J'}$ and $x\in \Omega_J$ with $J = J' +m$ and $m\ge 1$, we can bound 
the distance by 
$$
\mathrm{dist}(x,\alpha)\ge \ell 2^{\frac{J'+m}{1-\epsilon} - js} -  \ell 2^{\frac{J'}{1-\epsilon} - js}
\ge 
\ell 2^{-js} 2^{\frac{J'}{1-\epsilon}} \left(2^{\frac{m}{1-\epsilon}}-1\right)\ge 
\tfrac{1}{2}2^{-js} 2^{J'} 2^{\frac{m}{1-\epsilon}}.
$$
This allows us to bound $\frac{\mathrm{dist}(x,\alpha)}{2^{J'+1-js}}$ from below
by
$\tfrac{1}{4}2^{\frac{m}{1-\epsilon}}$; in turn,
$2^m \le 4^{1-\epsilon} \left(\frac{\mathrm{dist}(x,\alpha)}{2^{J'+1-js}}\right)^{1-\epsilon}$.
Thus,
$$\rho(x) = 
2^{J' +m +1-js} 
\le 4^{1-\epsilon} 2^{J'+1-js} \left(1 + \frac{\mathrm{dist}(x,\alpha)}{2^{J'+1-js}}\right)^{1-\epsilon}
 = 4^{1-\epsilon} \rho(\alpha)  \left(1 + \frac{\mathrm{dist}(x,\alpha)}{\rho(\alpha)}\right)^{1-\epsilon}.
$$
It follows that Theorems \ref{DeRoI} and \ref{lower} hold for approximation by $\phi_k$
and for this set of centers with local density $\rho$.

\begin{example}$\mbox{}$
In this example the global spacing is $h= 2^{-j}$, 
while the spacing near the origin will be $h^2=2^{-2j}$.
We choose to impose a slow growth condition with $\epsilon=1/3$. 
By Theorem \ref{DeRoI} we observe that the LD must have precision 
$\ell>2k/\epsilon$, so we choose $\ell = 7k$.
\begin{center}
\begin{figure}[h]
\includegraphics{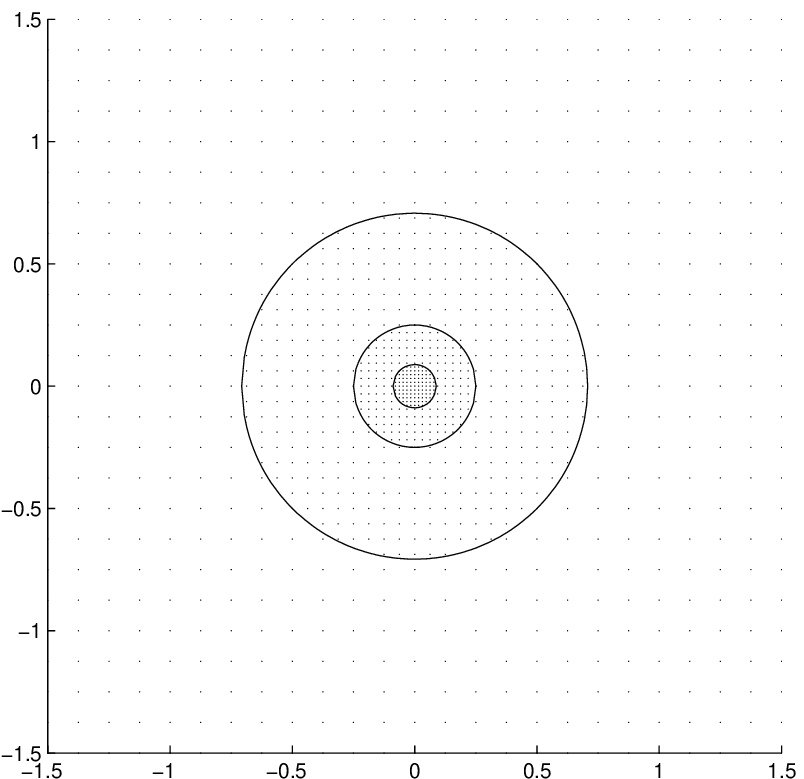}
\caption{A configuration of centers for thin plate spline approximation ($k=2$, as in the first remark) 
where the density ranges from  $\rho(0)\sim h^{2} =2^{-6}$ to
the coarsest density is $\rho(x)\sim h =2^{-3}$. }\label{Fi:bdry}
\end{figure}
\end{center}
The spacing  of centers immediately around the origin should be $2^{-2j}$, but there is an intermediate region where the spacing grows slowly. 
We decompose this in $j$ regions of increasing width: 
$$\Omega_{J} := \{x:  |x|\le 7k \times 2^{\frac{3J}{2}-2j}\}\setminus \bigcup_{0\le\ell<J} \Omega_{\ell},\quad J=1,\dots,\, j$$
in which we place gridded centers with spacing $h_J = 2^{J-2j}$. 
Thus, 
$\rho(x)\le 4^{2/3}\rho(y) (1+\frac{|x|}{\rho(y)})^{2/3}$ holds throughout $\reals^d$,
and by Theorem \ref{lower} there is a constant $C$ so that
for any $f$ with smoothness $0<s\le2k$ we have a good approximant with dramatically increased accuracy at the origin:
$$\|f - s_f\|_{\infty}\le C h^s \|f\|_{s}  \qquad \text{and} \qquad |f(0) - s_f(0)| \le C h^{2s}\|f\|_{s}.$$
In fact, for $x\in \Omega_{J}$, one has $|f(x) - s_f(x)|\le C 2^{s( J-2j)}\|f\|_{s}$. 

We note finally that the cardinality of centers in the ball $B(0,7k\times2^{-j/2})$ is less than
$$\sum_{J=0}^j (7k)^d2^{d(\frac{3J}{2}-2j)}2^{-d(J-2j)} \le(7k)^d \sum_{J=0}^j 2^{d(\frac{J}{2})}
\le C_d (7k)^d 2^{d(\frac{j}{2})}. $$
Since the same ball filled with centers having uniform spacing $h=2^{-j}$ holds roughly
$(7k)^d 2^{d(j/2)}$ centers, the increased precision comes at a cost of adding only a multiple of the original centers.
\end{example}


\section{Local Error Estimates for RBF Interpolation}
Many RBF interpolation results provide (implicitly or explicitly)
local error estimates similar to those of Theorem \ref{DeRoI}. In this section we discuss
two well known techniques for providing error estimates for RBF interpolation which 
happen to give estimates that are local.

In sharp contrast to the results of DeVore and Ron, only Definition \ref{LD} is needed for 
these estimates, and no extra global 
compatibility assumptions such as Definition \ref{SG} are needed for these (in other words, 
the estimates depend only on the nearest neighbors to point where the error is being measured).
The drawback comes from the fact that the estimates hold for only particular classes of target functions (i.e., for functions in the native space), and, although the error being measured is pointwise, the smoothness assumption on the target function is generally not measured in $L_{\infty}$.

This section is organized in two parts: first we discuss the localness inherent in Duchon's seminal 
result \cite{D1} concerning surface spline interpolation, 
and then we discuss the more general (in that it treats more kernels, and also addresses 
error of the derivative)  ``power function'' method.

\subsection{Local Error Estimates for Surface Spline Interpolation}
In this case we consider interpolation of a function $f$, defined initially on a bounded domain $\Omega$ with Lipschitz boundary, at a finite set of points $\Xi\subset \Omega$ using surface splines defined in
(\ref{surface}). That is, for $k>d/2$ we consider the unique interpolant to $f$ at $\Xi$ of the form
$$I_{\Xi}f (x)= \sum_{\xi\in\Xi} A_{\xi} \phi(x-\xi) +p(x),
\qquad
\text{with} \ p\in \Pi_{k-1}
\
\text{and}\ \sum_{\xi\in\Xi}A_{\xi}q(\xi) = 0\  \text{for all} \ q\in\Pi_{k-1}.$$
This is equivalent to finding the interpolant $s_{f,\Xi}$ in the Beppo-Levi space  
(also known as the homogeneous Sobolev space)
$D^{-k}L_2(\reals^d):=\{f\in C(\reals^d)\mid D^{\alpha}f\in L_2(\reals^d) \ \forall |\alpha|=k \}$ to $f$ at 
$\Xi$ that minimizes the Sobolev seminorm on $\reals^d$: 
$|u|_k :=|u|_{W_2^k(\reals^d)}$,
where the Sobolev seminorm defined on a measurable set $U$ is defined as
$$|u|_{W_2^k(U)}^2 :=   \int_{U} \sum_{|\alpha|=k}{k \choose \alpha} \left|D^{\alpha}u(x)\right|^2 \dif x.$$

Duchon's error estimates result from two key observations: 
\begin{enumerate}
\item a ``zeros lemma'' stating that any function defined on a ball $B$,
in the Sobolev class $W_2^k(B)$ (with $k>d/2$), having many zeros in $B$
 is necessarily small in $L_{\infty}(B)$, and  
\item an observation that the interpolation projector is bounded on $D^{-k}L_2$,
indeed, it is orthogonal with respect to the $D^{-k}L_2$ semi-inner product, with the
space of functions vanishing on $\Xi$ 
as its nullspace.
\end{enumerate}

\paragraph{1. Zeros Lemma} It is not  difficult, in this context, to derive a simplified, pointwise version of Duchon's zeros lemma \cite[Proposition 2]{D1}. Namely, for $k>d/2$ and $u\in W_2^k$ vanishing on a set $\Xi$ satisfying the conditions
of Definition \ref{LD} with precision $\ell \ge k-1$, we have
\begin{equation}\label{zeros}
|u(\alpha)|\le C \bigl(\rho(\alpha)\bigr)^{m-d/2} |u|_{W_2^k\bigl(B(\alpha,\rho(\alpha))\bigr)}.\end{equation}
To demonstrate this, we begin by observing that
$|u(\alpha)|\le (1+K)  \|p - u\|_{L_{\infty}\bigl(B(\alpha,\rho(\alpha)) \bigr)}$.
This follows directly from the fact that 
$|u(\alpha)| \le |p(\alpha)| +|u(\alpha) - p(\alpha) |$, for all $p\in \Pi_{\ell}$,  
and from Definition \ref{LD}:
$$|p(\alpha)| \le K \sup_{\xi\in \Xi\cap B(\alpha,\rho(\alpha))} |p(\xi)| 
\le 
K \sup_{\xi\in \Xi\cap B(\alpha,\rho(\alpha))} |p(\xi) - u(\xi)| 
\le 
K \|p - u\|_{L_{\infty}\bigl(B(\alpha,\rho(\alpha)) \bigr)}.$$
Estimate (\ref{zeros}) follows by dilating and translating a well-known result from polynomial approximation:
$\inf_{p\in \Pi_{\ell}}\|p- u\|_{L_{\infty}\bigl(B(0, 1) \bigr)} \le \inf_{p\in \Pi_{k-1}}\|p- u\|_{L_{\infty}\bigl(B(0, 1) \bigr)} \le C  |u|_{W_2^{k}\bigl(B(0,1) \bigr)}$.

\paragraph{2. Orthogonal Projection}By \cite[Lemma 3.2]{D1}, the interpolation operator 
$$I_{\Xi}: D^{-k}L_2(\reals^d) \mapsto
\left\{\sum_{\xi\in\Xi} A_{\Xi} \phi(\cdot - \xi) + p\mid p\in \Pi_{m-1}\ \text{and}\ \sum_{\xi \in \Xi} A_{\xi} q(\xi) = 0 
\text{ for all }q\in \Pi_{k-1}\right\}$$
satisfies the orthogonality relationship
$$|I_{\Xi}f|_{k}^2+|f- I_{\Xi}f|_{k}^2= |f|_{k}^2$$
for all $f\in D^{-k}L_2(\reals^d)$.

\paragraph{Local Error Estimate} At this point, we can apply the argument of \cite[Proposition 3]{D1}. We have
\begin{eqnarray*}
|f(\alpha) - I_{\Xi}f(\alpha)| &\le& 
C \bigl(\rho(\alpha)\bigr)^{k-d/2} |f - I_{\Xi}f |_{W_2^k\bigl(B(\alpha,\rho(\alpha))\bigr)}\\
&\le& C  \bigl(\rho(\alpha)\bigr)^{k-d/2} |f - I_{\Xi}f |_{k}\\
&\le&  C  \bigl(\rho(\alpha)\bigr)^{k-d/2} |f  |_{k}.
\end{eqnarray*}
We note that no extra conditions are necessary beyond Definition \ref{LD}, but this comes
with the drawback that the order is known only for functions in $W_2^k(\reals^d)$, and that this is an
unconventional smoothness space (for this problem), in that it measures smoothness in $L_2$ rather than $L_{\infty}$. In contrast,  Theorems \ref{DeRoI} and \ref{lower} seem to require Definition \ref{SG},
but they give orders in the range $s\in[0,2k]$ for functions from $C^{s}$.

\subsection{Local error estimates via the power function method}
An alternative method for measuring error, due originally to Wu and Schaback, estimates 
the pointwise error in terms of the `power function' associated with a conditionally positive 
definite kernel $\phi$,
%
\begin{eqnarray*}
\mathcal{P}_{\Xi}(x)&:=&
\left(\phi(0) - 2\sum_{\xi\in\Xi}\chi_{\xi}(x)\phi(x-\xi) +
\sum_{\xi,\zeta\in\Xi}  \chi_{\zeta}(x)\chi_{\xi}(x)\phi(\zeta-\xi)\right)^{1/2}.
\end{eqnarray*}
At each point $x$, it measures, roughly, the norm of the representer of the interpolation error at $x$:  $\delta_x (\mathrm{Id}-I_{\Xi})$ in a certain reproducing kernel semi-Hilbert space, the {\em native space}.
Our discussion is follows its development in the article of Wu and Schaback, \cite{WuSch}, where the local RBF error estimates we now discuss were first presented, and Wendland's text \cite{Wend}.

Consider RBF interpolation by a 
radial function that is conditionally positive definite of order $m$:
a radial function $\phi$ which, for any finite $\Xi\subset \reals^d$,
has a collocation matrix 
$\mathcal{C}_{\Xi}:=\bigl(\phi(\xi-\zeta)\bigr)_{\xi,\zeta}$ that is positive definite on vectors $\mathbf{A}\in \reals^{\Xi}$
satisfying 
\begin{equation}\label{vm}
\sum_{\xi\in\Xi} A_{\xi} p(\xi)=0 \qquad \text{for all} \quad p \in \Pi_{m-1}
\end{equation}
(we assume $\Pi_{-1}=\{0\}$ when $m=0$, in this case the kernel is simply positive definite).
For any finite, unisolvent $\Xi\in \reals^d$  (with respect to $\Pi_{m-1}$), 
the interpolation problem
possesses a unique solution 
$I_{\Xi} f$ in $S_m(\Xi):= \{\sum_{\xi\in\Xi} A_{\xi}\phi(\cdot-\xi) + p\mid p\in \Pi_{m-1}, \mathbf{A} \ \text{satisfying (\ref{vm})}\ \}.$

The associated native space
is a semi-Hilbert space $\mathcal{N}$ with a semi-inner product 
$\langle\cdot,\cdot\rangle_{\mathcal{N}}$ determined by $\phi$.
See \cite{Scha} or \cite[Chapter 8]{Wend} 
for a detailed construction.
The semi-inner product has $\Pi_{m-1}$ as its nullspace and there is a 
reproducing kernel, in the following sense:
\begin{equation}\label{Cond_Rep}
f(x) = (\mathrm{Proj_{{m-1}}} f)(x)+ \langle f,G(\cdot, x)\rangle_{\mathcal{N}} \qquad \text{for all}\ f\in \mathcal{N}.
\end{equation}
The operator $\mathrm{Proj_{{m-1}}}$ is a projector onto $\Pi_{m-1}$ with nullspace
determined by a fixed set of points $X$ poised for interpolation by $\Pi_{m-1}$. 
I.e., $X$ is a fixed, unisolvent set of points (with respect to $\Pi_{m-1}$)  with $\#X = \dim \Pi_{m-1}$.
The nullspace of  $\mathrm{Proj_{{m-1}}}$ is simply $
\left(\spam_{x_j\in X} \delta_{x_j}\right)\!\perp$, the joint kernel of the functionals $ \delta_{x_j} $  . 
For each $x$, the function $G(\cdot,x)$ (which  is uniquely determined by the projector
$\mathrm{Proj_{{m-1}}}$ and, in turn, by the fixed set $X$) is in $\mathcal{N}$, and reproduces 
the functional $\delta_{(x)}:=\delta_x - \delta_x \mathrm{Proj_{{m-1}}} $.

By expressing the interpolant in terms of the Lagrange basis, $I_{\Xi} f = \sum_{\xi\in\Xi} f(\xi) \chi_{\xi}$,  the interpolation error can be expressed, with the help of (\ref{Cond_Rep}) as
$$\left|f(x)-I_{\Xi} f(x)\right|  =\left| \langle f, G(\cdot,x)\rangle_{\N} + \mathrm{Proj_{{m-1}}}f(x)
- 
 \sum_{\xi\in\Xi}\bigl(\langle f, G(\cdot,\xi)\rangle_{\N} + \mathrm{Proj_{{m-1}}}f(\xi)\bigr)\chi_{\xi}(x)\right|.
$$
Since $\Pi_{m-1}\subset S_m(\Xi)$, $p\in \Pi_{m-1}$ can be written as 
$p = \sum_{\xi\in\Xi}p(\xi)\chi_{\xi}$, 
and we can simplify the above expression:
$$\left|f(x)-I_{\Xi} f(x)\right|  
 =
 \left| 
   \left\langle f, G(\cdot,x)
    - 
    \sum_{\xi\in\Xi} \chi_{\xi}(x) G(\cdot,\xi)
    \right \rangle_{\N} 
  \right|
  \le
   \left| f\right|_{\N}
   \left| G(\cdot,x)
    - 
    \sum_{\xi\in\Xi} \chi_{\xi}(x) G(\cdot,\xi)
    \right |_{\N}. 
$$

The quadratic form $Q_x(\mathbf{u}):= \phi(0) - 2\mathbf{u^{T}}R_{\Xi}(x) + \mathbf{u}^{T}\mathcal{C}_{\Xi}\mathbf{u}$, 
with $R_{\Xi}(x) =\bigl(\phi(x-\xi)\bigr)_{\xi\in\Xi}$,
is defined for 
$\mathbf{u}\in \reals^{\Xi}$, and one easily sees that 
$\mathcal{P}_{\Xi}(x)^2 = Q_x(\mathbf{u}^*)$
with $\mathbf{u^*} = (\chi_{\xi}(x))_{\xi\in\Xi}$.
 By  \cite[Lemma 11.3]{Wend}, one has for certain admissible vectors $\mathbf{u}$ -- namely for $\mathbf{u}\in \mathbb{M}_x:=\{\mathbf{u}\in\reals^{\Xi}\mid \sum_{\xi\in\Xi}u_{\xi} p(\xi)= p(x)\}$ that the
 quadratic form is related to the  function $G$ by 
$Q_x(\mathbf{u}) =  | G(\cdot,x)- 
    \sum_{\xi\in\Xi} u_{\xi} G(\cdot,\xi)
  |_{\N}^2. $
Therefore, it follows that $|f(x) - I_{\Xi}f(x)|\le \left| f\right|_{\N} \mathcal{P}_{\Xi}(x)$, which is \cite[Theorem 4]{WuSch}, and was alluded to in the first paragraph of this subsection. 
  
 We can say more, however, since the minimum of $Q_x(\mathbf{u})$ over 
  $\mathbb{M}_x$ is $ \mathcal{P}_{\Xi}(x)^2$, \cite[Theorem 1]{WuSch}. 
  It follows that one can estimate the power function at $x$ using $Q_x(\mathbf{u}),$ for any other
  $\mathbf{u}\in \mathbb{M}_x$.
  We choose  $\mathbf{u}$ determined by
  the coefficient kernel
   ${u}_{\xi} = a(\xi,x)$
obtained from Definition \ref{LD} (with precision $\ell \ge m-1)$.  In other words,
$$\bigl( \mathcal{P}_{\Xi}(x)\bigr)^2 \le \phi(0) - 2\sum_{\xi\in\Xi}a(\xi,x)\phi(x-\xi) +
\sum_{\xi,\zeta\in\Xi}  a(\zeta,x) a(\xi, x)\phi(\zeta-\xi)$$
Assume now that the RBF $\phi$ is in $C^s(\reals^d)$, $s\in (0,\infty)$ and that Definition \ref{LD} holds with polynomial precision $\ell$ where $\ell = \max(\lceil s \rceil,m)-1$.
Polynomial reproduction then gives, for any $p\in \Pi_{\ell}$:
\begin{eqnarray*}
\bigl( \mathcal{P}_{\Xi}(x)\bigr)^2 
&\le& \phi(x-x) - p(x-x) -  \sum_{\xi \in\Xi} a(\xi,x) \left(\phi(x-\xi) -
p(x-\xi)\right) \\
&\mbox{}& -\sum_{\xi\in\Xi} a(\xi,x) \left(\left(\phi(x-\xi) - p(x-\xi)\right)- \sum_{\zeta \in\Xi} a(\zeta,x) \left(\phi(\zeta-\xi) - p(\zeta-\xi)\right) \right)\\
&\le& (1+K) \|\phi - p\|_{L_{\infty}\bigl(0,\rho(x)\bigr)}+ (K+K^2) \|\phi - p\|_{L_{\infty}\bigl(0,2\rho(x)\bigr)}
\le C (1+K)^2 \bigl(\rho(x)\bigr)^s.
\end{eqnarray*}
It follows that for $f\in\mathcal{N}$  the interpolation error satisfies the pointwise bound:
$$|f(x) - I_{\Xi}f(x)| \le C (1+K)\bigl(\rho(x)\bigr)^{s/2} |f|_{\mathcal{N}}.$$

In this case, we note that, again, no extra conditions are necessary beyond Definition \ref{LD}, but
this comes with the drawback that the order is known only for functions in $\mathcal{N}$,
which typically  measures smoothness in $L_2$ rather than $L_{\infty}$.  In this case, however,
the power function approach has an extra advantage, which we have not discussed (treated in \cite{WuSch} and \cite[Chapter 11]{Wend}) in that it gives estimates for derivatives of the error, as well.

\bibliographystyle{siam}
\bibliography{LocalRBFSubmit}
\end{document}